\documentclass[12pt,a4paper]{article}
\usepackage{amssymb}
\usepackage{amsmath}
\usepackage{amsthm}
\usepackage{graphicx}
\usepackage[margin=0.5in]{geometry}
\usepackage{bm}
\usepackage{amsthm}
\usepackage{amscd}
\usepackage{amsbsy}
\usepackage{amsmath}
\usepackage{amsfonts}
\usepackage{amssymb}
\usepackage{calligra}
\usepackage{xcolor}
\usepackage{float}
\usepackage[T1]{fontenc}
\usepackage{multirow}
\usepackage{mathrsfs}
\usepackage{mathtools}
\usepackage{graphicx}
\usepackage{epstopdf}
\usepackage{subfigure}
\usepackage{booktabs}
\usepackage{listings}
\usepackage{longtable}
\usepackage{latexsym}
\usepackage{arydshln}
\usepackage{enumerate}
\usepackage{epsfig}
\usepackage{cite}
\usepackage{verbatim}
\usepackage{overpic}
\usepackage{abstract}
\allowdisplaybreaks[4]

\newtheorem {theorem*}{Theorem}
\newtheorem {theorem} {Theorem}

\newtheorem{lemma}{Lemma}
\newtheorem{corollary}{Corollary}

\numberwithin{equation}{section}
\numberwithin{lemma}{section}
\numberwithin{theorem}{section}
\numberwithin{proposition}{section}
\numberwithin{corollary}{section}

\usepackage[colorlinks,
            CJKbookmarks=true,
            linkcolor=blue,
            anchorcolor=red,
             urlcolor=blue,
            citecolor=blue
            ]{hyperref}

\begin{document}
\arraycolsep=1pt

\title{\Large\bf Toeplitz Operators with Positive Measures on Harmonic Fock Spaces
\footnotetext{\hspace{-0.35cm}
			\endgraf{\it E-mail: gx1780521400@163.com(Xue Gou)}
			\endgraf \hspace{1.1cm} {\it 763207251@qq.com(Xin Hu)}
			\endgraf \hspace{1.1cm} {\it huangsui@cqnu.edu.cn(Sui Huang) }
			\endgraf S. Huang was Supported by National Natural Science Foundation of China(Grant No.11501068), the Natrual Science Foundation of Chongqing (No.cstc2019jcyj-msxmX0295).}}
	\author{Xue Gou, Xin Hu, Sui Huang\thanks{Corresponding author} \\
		\small \em School of Mathematical Sciences, Chongqing Normal University, Chongqing, China}

\date{ }
\maketitle

\vspace{-0.8cm}

\begin{center}
\begin{minipage}{16cm}\small
{\noindent{\bf Abstract} \quad In this paper, we study the basic properties of Toeplitz Operators  with positive measures $\mu$ on harmonic Fock spaces. We prove equivalent conditions for boundedness, compactness and Schatten classes $S_{p}$ of $T_{\mu}$ by the methods of Berezin transform of operators.
\endgraf{\bf Mathematics Subject Classification (2010).}\quad  47B35, 47B10.
\endgraf{\bf Keywords.}\quad  Toeplitz operator; harmonic Fock space; harmonic Fock Carleson-type measures; boundedness; compactness; Schatten class.}
\end{minipage}
\end{center}

\section{Introduction}\label{s1}
 For any positive parameter $\alpha$, we consider the Gaussian measure
$$ d\lambda_{\alpha}(z)=\frac{\alpha}{\pi}e^{-\alpha|z|^{2}}dA(z),$$
 where $dA(z)=dxdy$ is the area measure on the complex plane $\mathbb{C}$. Suppose $r>0$,
we write $r\mathbb{Z}^{2}=\{nr+imr|n\in\mathbb{Z}, m\in\mathbb{Z}\}$ to denote lattices in complex plane $\mathbb{C}$,
and use
 $$S_{r}=\{z=x+iy| -\frac{r}{2}\leq x< \frac{r}{2}, -\frac{r}{2}\leq y< \frac{r}{2}\}$$
to represent the most basic small lattice in $r\mathbb{Z}^{2}$.
Obviously, the complex plane has the following divisions: $\mathbb{C}=\bigcup\{S_{r}+z|z\in r\mathbb{Z}^{2}\}$
and the area of $S_{r}$ is $A(S_{r})= r^{2}$, more details can be found in [1].
For any $p\geq1$, the space
$$L^{p}_{\alpha}(\mathbb{C}, dA)=\{f|\int_{\mathbb{C}}|f(z)e^{-\frac{\alpha|z|^{2}}{2}}|^{p}dA(z)<+\infty\}$$
denotes the space of all Lebesgue measurable functions $f$ on $\mathbb{C}$ such that $$f(z)e^{-\frac{\alpha}{2}|z|^{2}} \in L^{p}(\mathbb{C}, dA),$$ whose norm  is defined by
$$ ||f ||_{p,\alpha}=\{\frac{p\alpha}{2\pi}\int_{\mathbb{C}}|f(z)e^{-\frac{\alpha}{2}|z|^{2}}|^{p}dA(z)\}^{\frac{1}{p}}.$$
The Fock space $F^{p}_{\alpha}$ consists of all  entire functions $f(x)$ in $ L^{p}_{\alpha}(\mathbb{C}, dA)$,
 and we use $F^{p}_{h}$ to denote the spaces of all harmonic functions $f$ on $\mathbb{C}$ in $L^{p}_{\alpha}(\mathbb{C}, dA)$
 which are defined by Engli$\check{s}$ in [3] firstly with norm
 $$ ||f ||_{h,p,\alpha}=\{\frac{p\alpha}{2\pi}\int_{\mathbb{C}}|f(z)e^{-\frac{\alpha}{2}|z|^{2}}|^{p}dA(z)\}^{\frac{1}{p}}.$$
Specially, we write $F^{\infty}_{h}$ to denote the space of all harmonic functions $f$
which satisfy
$$ ||f||_{h,\infty, \alpha}=\textmd{ess sup}\{|f(z)|e^{-\frac{\alpha}{2}|z|^{2}}|z\in\mathbb{C}\}<+\infty .$$
We write  those norm as $||f ||_{h,p}$ for conveniece. When $p\geq1, $ $F^{p}_{\alpha}$ and $F^{p}_{h}$ are closed subspaces of $L^{p}_{\alpha}. $
And $F^{2}_{\alpha}$ and $F^{2}_{h}$ are Hilbert spaces with the inner product
$$\langle f, g\rangle=\int_{\mathbb{C}}f(z)\overline{g(z)}d\lambda_{\alpha}(z).$$
The reproducing kernels of $F^{2}_{\alpha}$ are $K_{z}(\omega)$=$e^{\alpha\bar{z}\omega}$,
and the normal reproducing kernels are $k_{z}(\omega)=e^{\alpha\bar{z}\omega-\frac{\alpha}{2}|z|^{2}}$.
\par
If $f$ belongs to $F^{p}_{h},$ we can write $f=f_{1}+\overline{f_{2}}, $
 where $f_{1}$ and $f_{2}$ are entire functions and which can be found in  [3]. Through a straight calculation,
\begin{align}
f(z)
 =&\langle f_{1},K_{z}\rangle+\overline{\langle f_{2},K_{z}\rangle}\\\nonumber
 =&\langle f_{1},K_{z}\rangle+\langle f_{1},\overline{K_{z}}\rangle+\langle\overline{f_{2}},K_{z}\rangle+\langle\overline{f_{2}},\overline{K_{z}}\rangle\\\nonumber
 =&\langle f_{1}+\overline{f_{2}},K_{z}+\overline{K_{z}}\rangle,
\end{align}
for any $f\in F^{2}_{h}$, so the reproducing kernels in $F^{2}_{h}$ are $H_{z}(\omega)=K_{z}(\omega)+\overline{K_{z}(\omega)}$ . And the normal reproducing kernels in $F^{2}_{h}$ are
$$h_{z}(\omega)=\frac{H_{z}(\omega)}{\sqrt{H_{z}(z)}}=\frac{K_{z}(\omega)+\overline{K_{z}(\omega)}}{\sqrt{2}\sqrt{K_{z}(z)}}
=\frac{\sqrt{2}}{2}(k_{z}(\omega)+\overline{k_{z}(\omega)}).$$
There is a orthogonal projection $P$ from $L_{\alpha}^{2}(\mathbb{C}, dA)$ onto $F^{2}_{h}$
which can be represented by
 $$(Pf)(z)=<f,H_z>=\int_{\mathbb{C}}f(\omega)\overline{H_{z}(\omega)}d\lambda_{\alpha}(\omega).$$
 This integral formula also holds for all $f\in L_\alpha^1(\mathbb{C}, dA)$.
 \par
Suppose $\mu$ is a positive Borel measure on $\mathbb{C}. $ Define Toeplitz operator with $\mu$ on $F_{h}^{2}$
as follows:
$$(T_{\mu}f)(z)=\int_{\mathbb{C}}f(\omega)\overline{H_{z}(\omega)}e^{-\alpha|\omega|^{2}}d\mu(\omega),$$
which is well defined. Let $f(\omega)=H_{z}(\omega)$, so $$(T_{\mu}H_{z})(z)=\int_{\mathbb{C}}|H_{z}(\omega)|^{2}e^{-\alpha|\omega|^{2}}d\mu(\omega).$$
If
$d\mu=\frac{\alpha}{\pi}\varphi dA$, where $\varphi\geq 0$ and satisfies that $\varphi f\in L_{\alpha}^{2}(\mathbb{C}, dA)$, then we can define Toeplitz operator with $\varphi$ on $F_{h}^{2}$
\begin{align}\nonumber
(T_{\varphi}f)(z)
=&\int_{\mathbb{C}}f(\omega)\overline{H_{z}(\omega)}e^{-\alpha|\omega|^{2}}\frac{\alpha}{\pi} \varphi(\omega)dA\\\nonumber
=&\int_{\mathbb{C}}f(\omega)\varphi(\omega)\overline{H_{z}(\omega)}d\lambda_{\alpha}(\omega)\\\nonumber
=&(PM_{\varphi}f)(z).
\end{align}
In this way, $T_{\varphi}$ agrees with the classic definition of Toeplitz operators on analytic function spaces.
\par
In the paper, we will study properties of $T_{\mu}$ by its  Berezin transform  as follows:
\begin{align}\nonumber
\widetilde{\mu}(z)=\widetilde{T}_{\mu}(z)
=&\langle T_{\mu}h_{z},h_{z}\rangle\\\nonumber
=&\int_{\mathbb{C}}(T_{\mu}h_{z})(\omega)\overline{h_{z}(\omega)}d\lambda_{\alpha}(\omega)\\\nonumber
=&\int_{\mathbb{C}}(\int_{\mathbb{C}}h_{z}(u)\overline{H_{\omega}(u)}e^{-\alpha|u|^{2}}d\mu(u))\overline{h_{z}(\omega)}
d\lambda_{\alpha}(\omega)\\\nonumber
=&\int_{\mathbb{C}}\overline{\int_{\mathbb{C}}h_{z}(\omega)H_{\omega}(u)d\lambda_{\alpha}(\omega)}h_{z}(u)e^{-\alpha|u|^{2}}d\mu(u)\\\nonumber
=&\int_{\mathbb{C}}\overline{\int_{\mathbb{C}}h_{z}(\omega)\overline{H_{u}(\omega)}
d\lambda_{\alpha}(\omega)}h_{z}(u)e^{-\alpha|u|^{2}}d\mu(u)\\\nonumber
=&\int_{\mathbb{C}}\overline{h_{z}(u)}h_{z}(u)e^{-\alpha|u|^{2}}d\mu(u)\\\nonumber
=&\int_{\mathbb{C}}|h_{z}(u)|^{2}e^{-\alpha|u|^{2}}d\mu(u).
\end{align}
In the same way, we can get that
\begin{align}\nonumber
\widetilde{\varphi}(z)=\widetilde{T}_{\varphi}(z)
=&\langle T_{\varphi}h_{z},h_{z}\rangle\\\nonumber
=&\langle \varphi h_{z},h_{z}\rangle\\\nonumber
=&\int_{\mathbb{C}}\varphi(\omega)h_{z}(\omega)\overline{h_{z}(\omega)}d\lambda_{\alpha}(\omega)\\\nonumber
=&\int_{\mathbb{C}}\varphi(\omega)|h_{z}(\omega)|^{2}d\lambda_{\alpha}(\omega).
\end{align}
\par
There are many similar results about boundedness and compactness of the  positive Toeplitz operators on
the Bergman space and Fock space. The Carleson-type measures on the Bergman space  is  studied firstly
by Hastings[4], and the positive Toeplitz operators on the
Bergman space in terms of Carleson-type measures first appears in McDonald and Sundberg[5],
Luecking[6] is the first one to study Toeplitz operators on the Bergman space with measures as
symbols, this idea is further pursued in Zhu[7]. Isralowitz and Zhu discussed the boundedness, compactness and $S_p$-class of Toeplitz operators $T_\mu$ with Carleson-type measures
as symbols on Fock space $F_\alpha^2$
in [8]. Wang study basic properties of Toeplitz operators $T_\mu$ on Fock-type space in [9].Stroefhoff used the Berezin transform of operators to
get characterizations about the properties on many kinds of classic spaces such as Hardy
space, Bergman space, and Fock space in [10], which covered the results of boundedness
and compactness of Toeplitz operators on Fock space studied by Coburn, Berger and Zhu in [11-14].
\par
Engli$\check{s}$ studied the properties of Berezin transform of operators on harmonic
function spaces including harmonic Fock space in [3]. Basing on the discussion by Engli$\check{s}$, we use the Berezin transform of operators to study the properties of Toeplitz Operators $T_\mu$ with Carleson-
type measures $\mu$ on harmonic Fock spaces, including characteristics of harmonic Fock-Carleson measures,
boundedness and compactness of $T_\mu$. Moreover, we obtain the necessary and sufficient condition for $S_p$-class of $T_\mu$.

\section{Harmonic Fock-Carleson Measures}
\noindent In this section, we will characterize Carleson-type measures for the harmonic Fock spaces.
First, we will estimate the norm of $h_{z}(\omega)$.
\begin{lemma}
 The norm of $h_{z}(\omega)$ is not more than $\sqrt{2}$ in $F_{h}^{p}$, where $ p\geq1$.
\end{lemma}
\begin{proof}
For $ p\geq1$, according to $\|k_{z}\|_{h,p}=1$
in \cite{ref1}, we have
\begin{align}\nonumber
\|h_{z}\|_{h,p}
= &\|\frac{\sqrt{2}}{2}(k_{z}+\overline{k_{z}})\|_{h,p}\\ \nonumber
\leqslant &\frac{\sqrt{2}}{2}(\|k_{z}\|_{h,p}+\|\overline{k_{z}}\|_{h,p})\\\nonumber
= &\sqrt{2}.
\end{align}
Specially when $p=2$, by Riesz representation theorem, we have $\|h_{z}\|_{h,2}=1.$
\end{proof}
\par We will discuss  pointwise estimate for functions in $F_{h}^{p}$ which should be used for later theorems.
\begin{lemma}
For any $r>0$ and $p\geqslant1$. If $f$ is a harmonic function, there exists a positive constant $C=C(p, \alpha, r)$ such that
$$|f(a)e^{-\frac{\alpha|a|^{2}}{2}}|^{p}\leqslant C\int_{B(a,r)}|f(\omega)e^{-\frac{\alpha|\omega|^{2}}{2}}|^{p}dA(\omega)$$
for all $z\in\mathbb{C}$.
\end{lemma}
\begin{proof}

Through a straight calculation,
\begin{align}\nonumber
&\int_{B(a,r)}|f(\omega)e^{-\frac{\alpha|\omega|^{2}}{2}}|^{p}dA(\omega)\\\nonumber
=&\int_{B(0,r)}|f(\omega+a)|^{p}e^{\frac{-p\alpha(|\omega+a|^{2})}{2}}dA(\omega)\\\nonumber
\geqslant &|f(a)|^{p}\int_{B(0,r)}e^{\frac{-p\alpha|\omega+a|^{2}}{2}}dA(\omega)\\\nonumber
=&|f(a)|^{p}\int_{B(0,r)}|e^{-\alpha \omega\overline{a}}|^{p}e^{\frac{-p\alpha(|\omega|^{2}+|a|^{2})}{2}}dA(\omega)\\\nonumber
\geqslant&|f(a)|^{p}{\int_{0}}^{r}{\int_{0}}^{2\pi}e^{\frac{-p\alpha(t^{2}+|a|^{2})}{2}}t dtd\theta\\\nonumber
=&C|f(a)|^{p}e^{\frac{-p\alpha|a|^{2}}{2}}
\end{align}
because $|f(\omega+a)|^{p},|e^{-\alpha\omega \overline{a}}|^{p}$ are subharmonic.
 This gives the desired result.
\end{proof}

\begin{lemma}
Suppose $\mu$ is a positive Borel measure on $\mathbb{C}$. For any $r>0$, exist $C=C_{r}$, such that
$$\mu(B(z,r))\leqslant C\widetilde{\mu}(z)$$
 for all $z\in\mathbb{C}$, where $B(z,r)=\{\omega\in\mathbb{C}||\omega-z|<r\}$.
\end{lemma}
\begin{proof}
For the definition of $\widetilde{\mu}(z)$, we have that
\begin{align}\nonumber
\widetilde{\mu}(z)
=&\int_\mathbb{C}|h_{z}(\omega)|^{2}e^{-\alpha|\omega|^{2}}d\mu(\omega)\\\nonumber
\geq&\int_{B(z,r)}|\frac{\sqrt{2}}{2}(k_{z}(\omega)+\overline{k_{z}(\omega)}|^{2}e^{-\alpha|\omega|^{2}}d\mu(\omega)\\\nonumber
=&\frac{1}{2}\int_{B(z,r)}|e^{\alpha\bar{z}\omega-\frac{\alpha|z|^2}{2}}+e^{\alpha{z}\bar{\omega}-\frac{\alpha|z|^2}{2}}|^{2}e^{-\alpha|\omega|^{2}}d\mu(\omega)\\\nonumber
=&\frac{1}{2}\int_{B(z,r)}|e^{\frac{-\alpha|\omega-z|^2}{2}}(e^{\frac{\alpha}{2}(\bar{z}\omega-z\bar{\omega})}+e^{\frac{\alpha}{2}(z\bar{\omega}-\omega\bar{z})})|^2d\mu(\omega)\\\nonumber
\geq&\frac{1}{2}e^{-\alpha r^2}\int_{B(z,r)}|e^{2\alpha i\textmd{Im}{\bar{z}\omega}}+1|^2d\mu(\omega)
\end{align}
Considering that the solutions to the equation $e^{2\alpha i\textmd{Im}{\bar{z}\omega}}+1=0$ are parallel finite line segments whose $\mu$-measure is 0, we can chose a constant $C_0>0$ and an open subset $E$ of $B(z, r)$ such that $\mu(E)\leq \frac{1}{2} \mu(B(z, r))$ and $|e^{2\alpha i\textmd{Im}{\bar{z}\omega}}+1|^2\geq C_0$ for $\omega \notin E.$ Then
\begin{align}\nonumber
\widetilde{\mu}(z)
\geq &\frac{1}{2}e^{\alpha r^{2}}\int_{B(z,r)-E}|e^{2\alpha i\textmd{Im}{\bar{z}\omega}}+1|^2d\mu(\omega)\\\nonumber
\geq &\frac{1}{2}e^{\alpha r^{2}}C_0\mu(B(z,r)-E)\\\nonumber
\geq &\frac{1}{4}C_0 e^{\alpha r^{2}} \mu(B(z,r)).
\end{align}
The result holds with $C=\frac{1}{4}C_0 e^{\alpha r^{2}}$.
\end{proof}
\par Next, we introduce the harmonic Fock-Carleson type measures. Let $\mu$ be a positive Borel
measure for harmonic Fock space $F_{h}^{p}$, and $p\geqslant1$. We say that $\mu$ is a harmonic Fock-Carleson type measure for $F_{h}^{p}$ if there exists a constant $C>0$ such that
$$\int_{\mathbb{C}}|f(z)e^{-\frac{\alpha|z|^{2}}{2}}|^{p}d\mu(z)
 \leqslant C\int_{\mathbb{C}}|f(z)e^{-\frac{\alpha|z|^{2}}{2}}|^{p}dA(z)$$
 for any $f$ belongs to $F_{h}^{p}$. It is clear that $\mu$ is a Carleson measure on $F_{h}^{p}$
if and only if  $F_{h}^{p}\subset L^p(\mathbb{C},d\mu)$ and the inclusion mapping
$$i_{p}: F_{h}^{p}\rightarrow L^{p}(\mathbb{C}, d\mu)$$
 is bounded.
The following  theorem is about equivalent conditions of harmonic Fock-Carleson type measures for $F_{h}^{p}$.
\begin{theorem}

Suppose $\mu$ is a positive Borel measure, $p\geqslant1, r>0$, and lattice in $\mathbb{C}$ is generated by $\{a_{n}\}$
with radius $r$. Then the following conditions are equivalent.\\
$(a)$ There exists a constant $C$, such that
 $\mu(B(z,r))\leqslant C$, for any $z\in\mathbb{C}$. Specially,
 $\mu(B(a_{n},r))\leqslant C$, for any $n\in\mathbb{Z}^{+}$.
 \\
 $(b)$  For all  functions $f\in {F_h}^p$, there exists a positive constant $C$ such that
 $$\int_{\mathbb{C}}|f(\omega)e^{-\frac{\alpha|\omega|^{2}}{2}}|^{p}d\mu(\omega)
 \leqslant C\int_{\mathbb{C}}|f(\omega)e^{-\frac{\alpha|\omega|^{2}}{2}}|^{p}dA(\omega) $$
 $(c)$ $\tilde{\mu}(z)$ is bounded on $\mathbb{C}$.
 \end{theorem}
 \begin{proof}.
 If condition $(a)$ holds, for any $f\in F_{h}^{p}$, we suppose
 \begin{align}\nonumber
 I(f)
 =&\int_{\mathbb{C}}|f(\omega)e^{-\frac{\alpha|\omega|^{2}}{2}}|^{p}d\mu(\omega)\\\nonumber
 \leqslant &\sum_{n=1}^{\infty}\int_{B(a_{n},r)}|f(\omega)e^{-\frac{\alpha|\omega|^{2}}{2}}|^{p}d\mu(\omega).
 \end{align}
By Lemma 2.2, we can get there exists a constant $C_{1}$, such that
$$|f(\omega)e^{-\frac{\alpha|\omega|^{2}}{2}}|^{p}\leqslant
C_{1}\int_{B(\omega,r)}|f(u)e^{-\frac{\alpha|u|^{2}}{2}}|^{p}dA(u).$$
By straight calculation, we have that
\begin{align}\nonumber
I(f)
\leqslant &\sum_{n=1}^{\infty}\int_{B(a_{n},r)}\{C_{1}\int_{B(\omega,r)}|f(u)e^{-\frac{\alpha|u|^{2}}{2}}|^{p}dA(u)\}d\mu(\omega)\\\nonumber
=&\sum_{n=1}^{\infty}C_{1}\int_{B(a_{n},2r)}|f(u)e^{-\frac{\alpha|u|^{2}}{2}}|^{p}dA(u)\int_{B(a_{n},r)}1d\mu(\omega).\nonumber
\end{align}
If condition $(a)$ holds, it means that $\int_{B(a_{n},r)}1d\mu(\omega)=\mu(B(a_{n},r))\leqslant C$,
then
\begin{align}\nonumber
I(f)
\leqslant &C_{2}\sum_{n=1}^{\infty}\int_{B(a_{n},2r)}|f(u)e^{-\frac{\alpha|u|^{2}}{2}}|^{p}dA(u)\\\nonumber
\leqslant &C_{2}N\int_{\mathbb{C}}|f(u)e^{-\frac{\alpha|u|^{2}}{2}}|^{p}dA(u),
\end{align}
for any $u\in\mathbb{C}$, $u$ is in at most $N$ lattices $B(a_{n},2r)$,
so condition $(b)$ holds.
To show that condition $(b)$ implies $(c)$, we consider the normalized reproducing kernals
$$h_{z}(w)=\frac{\sqrt{2}}{2}(k_{z}(\omega)+\overline{k_{z}(\omega)}).$$
By Lemma 2.3, we have
$$\widetilde{\mu}(z)=\int_{\mathbb{C}}|h_{z}(\omega)e^{-\frac{\alpha}{2}|\omega|^{2}}|^{2}d\mu(\omega)
\leq C \int_{\mathbb{C}}|h_{z}(\omega)e^{-\frac{\alpha}{2}|\omega|^{2}}|^{2}dA(\omega)=C
$$
 so that condition $(b)$ implies $(c)$ and condition $(c)$ can imply $(a)$.
\end{proof}

\par  Suppose $\mu$ is a Carleson measure on $F_{h}^{p}$, we say that $\mu$ is a vanishing harmonic Fock-Carleson measure for $F_{h}^{p}$ if the inclusion mapping $i_p$ above is compact, that is,
 $$\lim\limits_{n\rightarrow\infty}\int_{\mathbb{C}}|f_{n}(z)e^{-\frac{\alpha}{2}|z|^{2}}|^{p}d\mu(z)=0$$
 where $\{f_{n}\}$ is a bounded sequence in $F_{h}^{p}$ that
 converges to 0 uniformly on compact subsets of the complex plane $\mathbb{C}$.
The following  theorem is about equivalent conditions of vanishing harmonic Fock-Carleson type measures for $F_{h}^{p}$.
\begin{theorem}
Suppose $\mu$ is a positive Borel measure, $p\geqslant1, r>0,$ and the lattice in $\mathbb{C}$ is generated by $\{a_{n}\}$
with radius $r. $ Then the following conditions are equivalent:\\
$(a)$ $\mu$ is a vanishing Fock-Carleson measure.\\
$(b)$ $\mu(B(z, r))\rightarrow 0$ as $|z|\rightarrow +\infty$.\\
$(c)$ $\tilde{\mu}(z)\rightarrow 0$ as $|z|\rightarrow +\infty$.
\end{theorem}
\begin{proof}.
We omit the specific proof here, the same method had been used in theorem 2.1.
\end{proof}
\section{Boundedness and Compactness of $T_\mu$}
\noindent We will characterize the boundedness and compactness of Toeplitz operators $T_{\mu}$ with a positive Borel measure $\mu$ on $F_{h}^{2}$.
The boundedness is equivalent to that $\mu$ is a harmonic Fock-Carleson measure. The compactness is equivalent to that $\mu$ is a vanishing harmonic Fock-Carleson measure.
\begin{theorem}
For $\mu$ is a positive Borel measure on $\mathbb{C}$, the following conditions are equivalent.\\
$(a)$ The Toeplitz operator $T_{\mu}$ is bounded on $F_{h}^{2}$.\\
$(b)$ $\widetilde{\mu}(z)$ is bounded on $\mathbb{C}$.\\
$(c)$ The measure $\mu$ is a harmonic Fock-Carleson measure on $F_{h}^{2}$.
\end{theorem}
\begin{proof}
If $T_{\mu}$ is bounded on $F_{h}^{2}$ for any $f\in F_{h}^{2}$, from the integral representation of the inner product,
we have
$$\widetilde{\mu}(z)=\langle T_{\mu}h_{z}, h_{z}\rangle,~~~z\in \mathbb{C}. $$
By Cauchy-Schwarz inequality, the boundedness of $T_{\mu}$ on $F_{h}^{2}$ implies that
$$0 \leqslant \widetilde{\mu}(z)\leqslant \|T_{\mu}h_{z}\| \|h_{z}\|\leqslant 2\|T_{\mu}\|$$
for all $z\in \mathbb{C}$. So condition $(a)$ implies $(b)$.
\par
Because the condition $(b)$ is equivalent to $(c)$ in Theorem 2.1, if $p=2$, we can get
$$\int_{\mathbb{C}}|f(\omega)e^{\frac{-\alpha}{2}|\omega|^{2}}|^2d\mu(\omega)\leq C\int_{\mathbb{C}}|f(\omega)e^{-\frac{\alpha}{2}|\omega|^{2}}|^{2}dA(\omega)$$
which meaning The measure $\mu$ is a harmonic Fock-Carleson measure for $F_{h}^{2}$, then the condition $(b)$ implies $(c)$.
\par
If condition $(c)$ holds, $\tilde{\mu}$ is bounded, then there exists a positive constant $C$ such that
$$\int_{\mathbb{C}}|f(\omega)e^{\frac{-\alpha}{2}|\omega|^{2}}|^2d(\omega)\leq C\int_{\mathbb{C}}|f(\omega)e^{-\frac{\alpha}{2}|\omega|^{2}}|^{2}dA(\omega).$$
Now if $f$ and $g$ are in $F_{h}^{2}$,
\begin{align}\nonumber
|<T_\mu f,g>|
=&|\int_{\mathbb{C}}T_\mu f(z)\overline {g(z)}d\lambda_\alpha(z)|\\\nonumber
=&|\int_{\mathbb{C}}\int_{\mathbb{C}} f(\omega)\overline{H_z(\omega)}e^{-\alpha|\omega|^2}\overline {g(z)}d\mu(\omega)d\lambda_\alpha(z)|\\\nonumber
=&|\int_{\mathbb{C}} f(\omega)e^{-\alpha|\omega|^2}\{\int_{\mathbb{C}}\overline{H_z(\omega) g(z)}d\lambda_\alpha(z)\}d\mu(\omega)|\\\nonumber
=&|\int_{\mathbb{C}} f(\omega)\overline{g(\omega)}e^{-\alpha|\omega|^2}d\mu(\omega)|\\\nonumber
\leq&\int_{\mathbb{C}}|f(\omega)e^{\frac{-\alpha|\omega|^2}{2}}||\overline{g(\omega)}e^{\frac{-\alpha|\omega|^2}{2}}|d\mu(\omega)\\\nonumber
\leq&\{\int_{\mathbb{C}}|f(\omega)e^{\frac{-\alpha|\omega|^2}{2}}|^2d\mu(\omega)\}^\frac{1}{2}\{\int_{\mathbb{C}}|g(\omega)e^{\frac{-\alpha|\omega|^2}{2}}|^2 d\mu(\omega)\}^\frac{1}{2}\\\nonumber
\leq&C\parallel f\parallel_{h,2}\parallel g\parallel_{h,2},
\end{align}
we have that $T_{\mu}: F_{h}^{2}\rightarrow F_{h}^{2}$ is bounded, thus the condition $(c)$ implies $(a)$.
\end{proof}
\par Next, we will characterize the compactness of Toeplitz operators on $F_{h}^{2}$.
\begin{theorem}
For $\mu$ is a positive Borel measure on $\mathbb{C}$, the following conditions are equivalent: \\
$(a)$ $T_{\mu}$ is a compact Toeplitz operator on $F_{h}^{2}$.\\
$(b)$ $\widetilde{\mu}(z)\rightarrow 0$ as $|z|\rightarrow +\infty$.\\
$(c)$ $\mu$ is a vanishing Fock-Carleson measure.
\end{theorem}
\par The proof can be adapted from Theorem 3.1.
\section{Schatten                              Class $S_p$ of $T_\mu$}
\noindent In this section, we'll determine when a Toeplitz operator $T_{\mu}$ on $F_{h}^{2}$
belongs to the Schatten class $S_{p}$.
Related background information about Schatten class $S_{p}$ can be found in [15].
\par
 If $T_{\mu}$ is  positive and compact
on $F_{h}^{2}$, we can show that the trace of $T_{\mu}$ is
\begin{equation}
tr(T_{\mu})=\frac{2\alpha}{\pi}\int_{\mathbb{C}}\widetilde{T}_{\mu}(z)dA(z).
\end{equation}
Suppose $\{e_n(z)\}_{n=-\infty}^{+\infty}$ is the orthonormal basis of $F_{h}^{2}$, then the reproducing kernels of $F_{h}^{2}$ are $H_z(\omega)=\sum_{n=-\infty}^{+\infty}\overline{e_n(z)}e_n(\omega).$
For the definition of trace, we can get
\begin{align}\nonumber
tr(T_{\mu})
=&\sum_{n=-\infty}^{+\infty}\langle T_{\mu}e_n,e_n\rangle\\\nonumber
=&\sum_{n=-\infty}^{+\infty}\int_{\mathbb{C}}T_{\mu}e_n(z)\overline{e_n(z)}d\lambda_\alpha(z)\\\nonumber
=&\sum_{n=-\infty}^{+\infty}\int_{\mathbb{C}}\int_{\mathbb{C}}e_n(\omega)\overline{H_z}(\omega)
e^{-\alpha |\omega|^2}\overline{e_n}(z)d\mu(\omega)d\lambda_\alpha(z)\\\nonumber
=&\frac{\alpha}{\pi}\int_{\mathbb{C}}\int_{\mathbb{C}}|H_z(\omega)|^2e^{-\alpha |z|^2}e^{-\alpha |\omega|^2}dA(z)d\mu(\omega)\\\nonumber
=&\frac{2\alpha}{\pi}\int_{\mathbb{C}}\int_{\mathbb{C}}|h_z(\omega)|^2e^{-\alpha |\omega|^2}dA(z)d\mu(\omega)\\\nonumber
=&\frac{2\alpha}{\pi}\int_{\mathbb{C}}\widetilde{T}_{\mu}(z)dA(z)
\end{align}
for $h_z(\omega)=\sqrt{2}e^{-\frac{\alpha|z|^2}{2}}H_z(\omega).$
\par
As a result of (4.1), we get the trace formula for Toeplitz operators on harmonic Fock space $F_{h}^{2}$.
\begin{theorem}
Suppose $\mu$ is a positive Borel measure, then $T_{\mu}$ is in the trace class $S_{1}$
if and only if $\mu$ is finite on $\mathbb{C}$. Moreover,
$tr(T_{\mu})=\mu(\mathbb{C})$.
\end{theorem}
\begin{proof}  By Fubini's theorem, we have
\begin{align}\nonumber
tr(T_{\mu})
=&\frac{\alpha}{\pi}\int_{\mathbb{C}}\widetilde{T}_{\mu}(z)dA(z)\\\nonumber
=&\frac{\alpha}{\pi}\int_{\mathbb{C}}<T_{\mu}h_{z},h_{z}>dA(z)\\\nonumber
=&\frac{\alpha}{\pi}\int_{\mathbb{C}}\int_{\mathbb{C}}|h_{z}(u)|^{2}e^{-\alpha|u|^{2}}d\mu(u)dA(z)\\\nonumber
=&\frac{\alpha}{\pi}\int_{\mathbb{C}}e^{-\alpha|u|^{2}}\int_{\mathbb{C}}|h_{z}(u)|^{2}dA(z)d\mu(u)\\\nonumber
=&\int_{\mathbb{C}}e^{-\alpha|u|^{2}}\int_{\mathbb{C}}|h_{z}(u)|^{2}e^{\alpha|z|^{2}}d\lambda_{\alpha}(z)d\mu(u)\\\nonumber
=&\int_{\mathbb{C}}\frac{h_{u}(u)}{\sqrt{H_{u}(u)}}d\mu(u)\\\nonumber
=&\mu(\mathbb{C}).
\end{align}
This also shows that $tr(T_{\mu})<+\infty$ if and only if $\mu(\mathbb{C})<+\infty$.
\end{proof}
\par
If $\mu$ is a locally finite Borel
measure, We define
$$\widehat{\mu}_{r}(z)=\mu(B(z,r))/(\pi r^{2}), ~~~z\in\mathbb{C},$$
then we have the following Theorem.
\begin{lemma}
Suppose $r>0$, $\mu$ is a positive Borel measure on $\mathbb{C}, p\geq 1$.
If $\widehat{\mu}_{r}\in L^{p}(\mathbb{C},dA)$, then $T_{\mu}$ and $T_{\widehat{\mu}}$ are bounded on $F_{h}^{2}$.
Moreover, there exists a positive constant $C$ (independent of $\mu$) such that $T_{\mu}\leqslant C T_{\widehat{\mu}_{r}}$.
\end{lemma}
\begin{proof}
Since $\widehat{\mu}_{r}\in L^{p}(\mathbb{C},dA)$, by Theorem 3.1, we can prove that  $T_{\mu}$ and $T_{\hat{\mu}_{r}}$ are bounded.
So if $f\in F_{h}^{2}$, we use Fubini's theorem to obtain that
\begin{align}\nonumber
\langle T_{\hat{\mu}_{r}}f, f \rangle
=& \int_{\mathbb{C}}T_{\hat{\mu}_{r}}f(z)\overline{f(z)}d\lambda_{\alpha}(z)\\\nonumber
=& \int_{\mathbb{C}}\hat{\mu}_{r}(z)|f(z)|^{2}d\lambda_{\alpha}(z)\\\nonumber
=& \frac{1}{\pi r^{2}}\int_{\mathbb{C}}\mu(B(z, r))|f(z)|^{2}d\lambda_{\alpha}(z)\\\nonumber
=& \frac{1}{\pi r^{2}}\int_{\mathbb{C}}|f(z)|^{2}[\int_{B(z, r)}d\mu(\omega)]d\lambda_{\alpha}(z)\\\nonumber
=& \frac{1}{\pi r^{2}}\int_{\mathbb{C}}|f(z)|^{2}d\lambda_{\alpha}(z)\int_{\mathbb{C}}  \chi_{B(z, r)}(\omega)d\mu(\omega)\\\nonumber
=& \frac{1}{\pi r^{2}}\int_{\mathbb{C}}d\mu(\omega)\int_{\mathbb{C}}|f(z)|^{2} \chi_{B(z, r)}(\omega)d\lambda_{\alpha}(z)\\\nonumber
=&\frac{\alpha}{\pi^2 r^{2}}\int_{\mathbb{C}}d\mu(\omega)\int_{B(\omega, r)}|f(z)e^{-\frac{\alpha}{2}|z|^{2}}|^{2}dA(z).
\end{align}
According to Lemma 2.2, there exists a positive constant $C$ such that
$$ \langle T_{\mu}f, f\rangle=\int_{\mathbb{C}}|f(\omega)|^{2}e^{-\alpha|\omega|^{2}}d\mu(\omega) \leq C\langle T_{\hat{\mu}_{r}}f, f \rangle $$
\end{proof}
which shows that $T_\mu\leq C T_{\hat{\mu}_r}.$
\par
Next, we will give several equivalent conditions about the classification of $T_{\mu}$.
\begin{theorem}
Suppose $\mu$ is a positive Borel measure, $p\geqslant1, r>0$, and the lattice in $\mathbb{C}$ is generated by $\{a_{n}\}$
with radius $r$. Then the following conditions are equivalent:\\
$(a)$ The operator $T_{\mu}$ is in Schatten class $S_{p}$.\\
$(b)$ The function $\widetilde{\mu}_{r}(z)$ is in $L^{p}(\mathbb{C}, dA)$.\\
$(c)$ The function $\hat{\mu}_r(z)$ is in $L^{p}(\mathbb{C}, dA)$.\\
$(d)$ The sequence $\{\mu(B(a_{n},r))\}$ is in $l^{p}$.
\end{theorem}
\begin{proof}
Suppose $p\geqslant1$ and  $T$ is a positive operator on a Hilbert space $H$. if $T\in S_{p}$, then $T^{p}\in S_{1}$ and $\langle Tx, x\rangle^{p} ~\leq~  \langle T^{p}x, x\rangle$ for $x\in H $ in \cite{ref15}, we have
$$(\tilde{\mu}(z))^{p} \leqslant \widetilde{T_{\mu}^{p}}(z)$$
because for $f\in {F_h}^2$,
$$\langle T_\mu f, f\rangle^{p} ~\leq~  \langle{T_\mu}^{p}f, f\rangle.$$
So
$$\int_{\mathbb{C}}|\tilde{\mu}(z)|^{p}dA(z) \leqslant \int_{\mathbb{C}}|\widetilde{T_{\mu}^{p}}(z)|dA(z)=tr(T_{\mu}^{p})<+\infty,$$
so condition $(a)$ implies $(b)$. Lemma 2.3 shows that condition $(b)$ implies $(c)$.
\par If condition $(c)$ holds, suppose $\widehat{\mu}_{r}(z) \in L^{\infty}(\mathbb{C}, dA)$, then
\begin{align}\nonumber
|\langle T_{\widehat{\mu}_{r}}e_n,e_n\rangle|
=&\int_{\mathbb{C}}\widehat{\mu}_{r}(z)|e_n(z)|^2d\lambda_\alpha(z)\\\nonumber
\leq & \parallel\widehat{\mu}_{r}\parallel_\infty \int_{\mathbb{C}}|e_n(z)|^2d\lambda_\alpha(z)\\\nonumber
=&\parallel\widehat{\mu}_{r}\parallel_\infty
\end{align}
where $\{e_n(z)\}_{n=-\infty}^{+\infty}$ is the orthonormal basis of $F_{h}^{2}$, thus $T_{\widehat{\mu}_{r}}\in S_\infty$.
If $\widehat{\mu}_{r}(z) \in L^{1}(\mathbb{C}, dA)$,
\begin{align}\nonumber
\sum_{n=-\infty}^{+\infty}|\langle T_{\widehat{\mu}_{r}}e_n,e_n\rangle|
=&\sum_{n=-\infty}^{+\infty}\int_{\mathbb{C}}\widehat{\mu}_{r}(z)|e_n(z)|^2d\lambda_\alpha(z)\\\nonumber
=&\frac{\alpha}{\pi} \int_{\mathbb{C}}\widehat{\mu}_{r}(z)(\sum_{n=-\infty}^{+\infty}|e_n(z)|^2) e^{-\alpha|z|^2}dA(z)\\\nonumber
=&\frac{\alpha}{\pi} \int_{\mathbb{C}}\widehat{\mu}_{r}(z)H_z(z) e^{-\alpha|z|^2}dA(z)\\\nonumber
=&\frac{2\alpha}{\pi}\parallel\widehat{\mu}_{r}\parallel_1
\end{align}
then we get $T_{\widehat{\mu}_{r}}\in S_1$. By interpolation, $T_{\widehat{\mu}_{r}}\in S_p,$
 when $\widehat{\mu}_{r}\in L^{p}(\mathbb{C}, dA).$ And by $T_{\mu}\leq C T_{\widehat{\mu}_{r}}$ for
 some $C>0$, $T_{\mu}\in S_p.$ So condition $(a)$ holds.
\par
We will prove that condition $(d)$ is equivalent to the other conditions. We first suppose that condition $(c)$ holds, which implies that
the function $\mu(B(z, 2r))$ is in $L^{p}(\mathbb{C}, dA)$. Choose a positive integer $N$ such that each point in the complex plane belongs to
at most $N$ of the discs $B(a_{n}, r)$. When $1\leqslant p< +\infty$,
$$\sum_{n=1}^{\infty}\int_{B(a_{n}, r)}\mu(B(z, 2r))^{p}dA(z)\leq N\int_{\mathbb{C}}\mu(B(z, 2r))^{p}dA(z).$$
For each $z\in B(a_{n}, r)$ we deduce from the triangle inequality that
 $$ \mu(B(a_{n},r))\leq \mu(B(z, 2r)),$$
 Thus
 $$ \pi r^{2}\sum_{n=1}^{\infty}\mu(B(a_{n}, r))^{p}\leq N\int_{\mathbb{C}}\mu(B(z, 2r))^{p}dA(z).$$
 When  $\widehat{\mu_{r}}(z)$ is in $L^{\infty}(\mathbb{C}, dA)$ for each $z\in B(a_{n}, r)$ we have
 $$sup|\mu(B(a_{n},r))|\leqslant sup|\mu(B(z, 2r))|< +\infty.$$
 This shows that condition $(c)$ implies $(d)$.
 \par Conversely, we assume that condition $(d)$ holds, that is
 $$\sum_{n=1}^{\infty}\mu(B(a_{n}, r))^{p}< +\infty$$
 for any $1\leqslant p< +\infty$. When $z_{n}=\frac{mr}{2}+\frac{nri}{2}$, it is easy to see $$\sum_{n=1}^{\infty}\mu(B(z_{n}, r))^{p}< +\infty.$$
 In fact, for each point $z_{k}$ that is not in the lattice $\{a_{n}\}$, the disc $B(z_{k}, r)$ is covered by six adjacent discs $B(z_{k}, r)$.
 So
\begin{align}\nonumber
\int_{\mathbb{C}}\mu(B(z, \frac{r}{2}))^{p}dA(z)
\leqslant &\sum_{n=1}^{\infty}\int_{B(z_{n}, \frac{r}{2})}\mu(B(z, \frac{r}{2}))^{p}dA(z)\\\nonumber
\leqslant &\sum_{n=1}^{\infty}\int_{B(z_{n}, \frac{r}{2})}\mu(B(z_{n}, r))^{p}dA(z)\\\nonumber
=&\frac{\pi r^{2}}{4}\sum_{n=1}^{\infty}\mu(B(z_{n}, r))^{p}\\\nonumber
=&\frac{6\pi r^{2}}{4}\sum_{n=1}^{\infty}\mu(B(a_{n}, r))^{p}<\infty.
\end{align}
When $p=+\infty$, we obtain the result by Theorem 2.1.
This shows that condition $(d)$ implies $(c)$. Now we have completed the proof of this theorem.
\end{proof}
\par
Suppose the nonnegative function $\varphi \in L^1(\mathbb{C},dA)$ and $r>0$, we define a new function
$$\hat{\varphi}_r(z)=\frac{1}{\pi\alpha}\int_{B(z, r)} \varphi(\omega)dA(\omega)$$
on $\mathbb{C}$. The following results are immediate consequence of  Theorem 3.1, 3.2 and Theorem 4.2.
\begin{corollary}
If $\varphi$ is a nonnegative function in $L^1(\mathbb{C},dA)$ and $r>0$, then
following conditions are equivalent:\\
$(a)$ $T_\varphi$ is bounded on $F_h^2$;\\
$(b)$ $\tilde{\varphi}(z)$ is bounded on $\mathbb{C}$;\\
$(c)$ $\hat{\varphi}_r(z)$ is bounded on $\mathbb{C}$.
\end{corollary}
\begin{corollary}
If $\varphi$ is a nonnegative function in $L^1(\mathbb{C},dA)$ and $r>0$, then
following conditions are equivalent:\\
$(a)$ $T_\varphi$ is compact on $F_h^2$;\\
$(b)$ $\tilde{\varphi}(z)\rightarrow 0$ as $|z|\rightarrow +\infty$;\\
$(c)$ $\hat{\varphi}_r(z)\rightarrow 0$ as $|z|\rightarrow +\infty$.
\end{corollary}
\begin{corollary}
If $\varphi$ is a nonnegative function in $L^1(\mathbb{C},dA)$ and $r>0$,$1\leq p\leq +\infty$, then
following conditions are equivalent:\\
$(a)$ $T_\varphi$ is in the Schatten class $S_p$;\\
$(b)$ $\tilde{\varphi}(z)$ is in  $L^p(\mathbb{C},dA)$;\\
$(c)$ $\hat{\varphi}_r(z)$ is in  $L^p(\mathbb{C},dA)$.
\end{corollary}
\par
Compared with the property of $T_{\mu}$, Toeplitz operators $T_\varphi$ induced by $\varphi$ can be classified into the Schatten class $S_{p}$ by the following theorem.
\begin{theorem}
If  $\varphi $ is a nonnegative function in $L^{p}(\mathbb{C}, dA)$, $p\geqslant1$, then $T_{\varphi}\in S_{p}$.
\end{theorem}
\begin{proof}
For any $f\in F_{h}^{2}$, if we assume that $\varphi$ has compact support in $\mathbb{C}$, then $T_{\varphi}$ is a positive compact operator whose canonical decomposition
is of the form
$$T_{\varphi}f=\sum_{n=-\infty}^{+\infty}\lambda_{n}\langle f, e_{n}\rangle e_{n},$$
where $\{\lambda_{n}\}$ is the eigenvalue sequence of $T_{\varphi}$ and $\{e_{n}\}$ is an orthonormal basis in $F_{h}^{2}$. In particular,
$$\lambda_{n} = \langle T_{\varphi}e_{n}, e_{n}\rangle= \int_{\mathbb{C}}|e_{n}(z)|^{2}\varphi(z)d\lambda_{\alpha}(z)$$
for all $n \in \mathbb{Z}$. When $1\leqslant p < +\infty$, it follows from H$\ddot{o}$lder inequality that
$$|\lambda_{n}|^{p} \leqslant \int_{\mathbb{C}}|e_{n}(z)|^{2}|\varphi(z)|^{p}d\lambda_{\alpha}(z),$$
then by proposition 7.11 in \cite{ref15}, we have
\begin{align}\nonumber
\sum_{n=-\infty}^{+\infty}|\lambda_{n}|^{p}
\leqslant &\int_{\mathbb{C}}(\sum_{n=-\infty}^{+\infty}|e_{n}(z)|^{2})|\varphi(z)|^{p}d\lambda_{\alpha}(z)\\\nonumber
= &\int_{\mathbb{C}}|\varphi(z)|^{p}H_{z}(z)d\lambda_{\alpha}(z)\\\nonumber
= &\frac{2\alpha}{\pi}\int_{\mathbb{C}}|\varphi(z)|^{p}dA(z).
\end{align}
When $\varphi\in L^{p}(\mathbb{C}, dA)$, $$\|T_{\varphi}\|_{S_{p}}^{p} \leqslant \frac{2\alpha}{\pi}\int_{\mathbb{C}}|\varphi(z)|^{p}dA(z) < +\infty.$$
When $p=+\infty,$
\begin{align}\nonumber
|\lambda_{n}|
\leqslant &\int_{\mathbb{C}}|e_{n}(z)|^{2}|\varphi(z)|d\lambda_{\alpha}(z)\\\nonumber
\leqslant &\|\varphi\|_{\infty}\|e_{n}\|_{2}\\\nonumber
=&\|\varphi\|_{\infty}.
\end{align}
Thus, $\{\lambda_{n}\}\in l^{\infty}$, that is $T_{\varphi}\in S_{\infty}$.
\par If $\varphi$ hasn't compact support in $\mathbb{C}$, then we let $\varphi_{n}=\varphi\chi_{n}$,
where
\begin{equation}\nonumber
\chi_{n}(z)=
\begin{cases}
1, & n-1<|z|\leqslant n,\\
0, & |z|\leqslant n-1 ~~or~~ |z|>n .
\end{cases}
\end{equation}
For $\varphi_{n}$, the Theorem holds, $\varphi=\sum\limits_{n=1}^{+\infty}\varphi_{n}$.
The desired result is proved.
\end{proof}
\par
\textbf{Declarations}: In our research, X.Gou and S.Huang calculate and prove results in the paper together. X.Gou drafts the article and S.Huang revise it carefully. S.Huang is supported by National Natural Science Foundation of China(Grant No.11501068)and the Natrual Science Foundation of Chongqing (No.cstc2019jcyj-msxmX0295).
\par
\textbf{Declarations of interest}: none.


\begin{thebibliography}{00}
 \vspace{-0.3cm}
\bibitem{ref1}
 K.H.Zhu, Analysis on Fock spaces. Grad. Texts Math., vol. 263, Springer, New York, 2012.
 
 \vspace{-0.3cm}
\bibitem{ref2} 
S.Janson, J.Peetre, R.Rochberg, Hankel forms and the Fock space, Rev. Mat. Iberoamericana 3 (1987) 61-138.

\vspace{-0.3cm}
\bibitem{ref3}
 M.Engli$\check{s}$,  Berezin transform on the harmonic Fock space, J. Math. Ana. App. 367(1) (2009) 75-97.
 
 \vspace{-0.3cm}
 \bibitem{ref4}
 W.W.Hastings, A Carleson measure theorem for Bergman spaces, Proc. Amer. Math. Soc. 52 (1975) 237-241.

\vspace{-0.3cm}
 \bibitem{ref5} 
 D.McDonald, C.Sundberg,Toeplitz operators on the disk, Indiana Univ. Math. J. 28 (1979) 595-611.
 
 \vspace{-0.3cm}
 \bibitem{ref6}
  D.H.Luecking, Trace ideal criteria for Toeplitz operators, J. Funct. Anal. 73(2) (1987) 345-368.
 
 \vspace{-0.3cm}
 \bibitem{ref7}
 K.H.Zhu, Positive  Toeplitz operators weighted Bergman spaces of bounded symmetroic domains, J. Operator
 Theory 20 (1988) 329-357.

\vspace{-0.3cm}
\bibitem{ref8}
 J.Isralowitz, K.H.Zhu,  Toeplitz operators on the Fock space, Integr. Equ. Oper. Theory 66(4) (2010) 593-611.

\vspace{-0.3cm}
\bibitem{ref9}X.F.Wang, Z.H.Tu, Z.J.Hu, Bounded and compact Toeplitz operators with positive measure symbols on  Fock-type  spaces, Acta Math. Sci. Ser. B(Engl. Ed.)40(3)(2020) 625-640.

\vspace{-0.3cm}
\bibitem{ref10} 
K.Stroethoff, The Berezin transform and operators on spaces of analytic functions, Linear operators Banach center publications 38 (1997) 361-380.

\vspace{-0.3cm}
\bibitem{ref11}
 C.Berger, L.Coburn, Toeplitz operators and quantum mechanics, J. Funct. Anal. 68(3) (1986) 273-299.

\vspace{-0.3cm}
\bibitem{ref12} 
C.Berger, L.Coburn, Toeplitz operators on the Segal-Bargmann space, Tran. Am. Math. Soc. 301 (1987) 813-829.

\vspace{-0.3cm}
\bibitem{ref13}
 C.Berger, L.Coburn, Heat Flow and Berezin-Toeplitz Estimates, Am. J. Math. 116(3) (1994) 563-590.

\vspace{-0.3cm}
\bibitem{ref14} 
C.Berger, L.Coburn, K.H.Zhu, Function theory on artan domains and Berezin-Toeplitz symbol calculus, Am. J. Math. 110 (1988) 921-953.

\vspace{-0.3cm}
\bibitem{ref15}  
K.H.Zhu, Operator Theory in Function Spaces, Math. Surv. Monogr., vol. 138, American Mathematical Society, Providence, RI, 2007.

\end{thebibliography}
\end{document}